\documentclass[12pt,a4paper]{article}
\usepackage{amsmath,amsthm,amssymb,amsfonts,stmaryrd}

\newcommand{\C}{{\mathbb{C}}}          
\newcommand{\R}{{\mathbb{R}}}          

\newcommand{\gdois}{{\mathrm{G}_2}}

\newcommand{\SO}{{\mathrm{SO}}}
\newcommand{\SU}{{\mathrm{SU}}}
\newcommand{\uni}{{\mathrm{U}}}

\newcommand{\rr}{\rightarrow}
\newcommand{\lrr}{\longrightarrow}

\newcommand{\calc}{{\cal C}}             %
\newcommand{\cale}{{\cal E}}             %
\newcommand{\calf}{{\cal F}}             %
\newcommand{\cali}{{\cal I}}             %
\newcommand{\calj}{{\cal J}}             %
\newcommand{\calri}{{{\cal R}^\xi}}             %
\newcommand{\cals}{{\cal S}}             %

\newcommand{\na}{{\nabla}}

\newcommand{\dx}{{\mathrm{d}}}

\newcommand{\inv}[1]{{#1}^{-1}}

\newcommand{\cinf}[1]{{\mathrm{C}}^\infty_{#1}}

\newcommand{\vol}{{\mathrm{vol}}}

\newcommand{\ric}{{\mathrm{Ric}\,}}
\newcommand{\scal}{{\mathrm{scal}\,}}

\newtheorem{teo}{Theorem}[section]
\newtheorem{lemma}{Lemma}[section]
\newtheorem{coro}{Corollary}[section]
\newtheorem{prop}{Proposition}[section]

\def\cyclic{\mathop{\kern0.9ex{{+}
\kern-2.2ex\raise-.28ex\hbox{\Large\hbox{$\circlearrowright$}}}}\limits}

\title{A fundamental differential system of 3-dimensional Riemannian geometry}

\author{R. Albuquerque}

\begin{document}


\maketitle


\begin{abstract}

We briefly recall a fundamental exterior differential system introduced by the author and then apply it to the case of three dimensions. Here we find new global tensors and intrinsic invariants of oriented Riemaniann 3-manifolds. The system leads to a remarkable Weingarten type equation for surfaces on hyperbolic 3-space. An independent proof for low dimensions of the structural equations gives new insight on the intrinsic exterior differential system.

\end{abstract}


\ 
\vspace*{3mm}
\vspace*{2mm}

\markright{\sl\hfill  R. Albuquerque \hfill}

\vspace*{6mm}


\section{A fundamental differential system}
\label{sec:Thedifferentialsystem}

This article presents the fundamental exterior differential system of Riemannian geometry introduced in \cite{Alb2011arxiv}, now applied in the 3-dimensional case. We start with some words in the form of introduction.

The intrinsic structure found in \cite{Alb2011arxiv} consists, in general, on a natural set of differential forms $\alpha_0,\ldots,\alpha_n$ existing on the total space $\cals$ of the unit tangent sphere bundle $SM\lrr M$ of any given oriented Riemannian $n+1$-manifold $M$. It is well-known that $\cals$ is a contact Riemannian manifold with the Sasaki metric. 

The theory applied to Riemannian surfaces is classical, as we shall recall next, considering the case $n=1$. Indeed, the famous structural equations due to Cartan give a {global} coframing on $\cals$, the total space of the tangent circle bundle over a surface $M$, with contact 1-form $\theta$ and two 1-forms $\alpha_0$ and $\alpha_1$. Denoting by $c$ the Gauss curvature of $M$, we find the following equations e.g. in \cite[pp. 168--169]{SingerThorpe}:
\begin{equation}\label{dalphais_m=2}
 \dx\alpha_0=\theta\wedge\alpha_1,\qquad \dx\alpha_1=c\,\alpha_0\wedge\theta,\qquad
   \dx\theta=\alpha_1\wedge\alpha_0   .
\end{equation}
Certainly $c$ is a constant along the $S^1$-fibres and no more than that in general.

Now let us see the case $n=2$ and hence assume $\cals$ is the total space of the unit tangent sphere bundle of an oriented Riemannian 3-dimensional manifold $M$. Then on $\cals$ we have again a contact 1-form $\theta$ and four pairwise orthogonal 2-forms $\alpha_0,\ \alpha_1,\ \alpha_2$ and $\dx\theta$, satisfying: 
\begin{equation}\label{introd_derivadasdastres2formas}
\begin{split}
  *\theta=\alpha_0\wedge\alpha_2=-\frac{1}{2}\,\alpha_1\wedge\alpha_1=-\frac{1}{2}\,\dx\theta\wedge\dx\theta , \\
  \dx\alpha_0=\theta\wedge\alpha_1,  \qquad\quad
 \dx\alpha_2  = \calri\alpha_2  , \hspace{6mm} \\
  \dx\alpha_1=2\,\theta\wedge\alpha_2-r\,\theta\wedge\alpha_0 .  \hspace{12mm}  
\end{split}
\end{equation}
The function $r=r(u)=\ric(u,u),\ u\in\cals$, and the 3-form $\calri\alpha_2$ are curvature dependent tensors. E.g. for constant sectional curvature $c$, we have $r=2c$ and $\calri\alpha_2=-c\,\theta\wedge\alpha_1$.

The differential system in general dimension interacts with various Euler-Lagrange systems of hypersurface equations of $M$, when we consider the theory in parallel with the Euclidean case described in \cite{BGG}. In the present setting, we focus on results on the intrinsic Riemannian geometry of $M$, since the differential system is fairly new. In dimension 3 the equations satisfy a coincidence that the $\alpha_i$ are 2-forms like $\dx\theta$, and so a natural $\SU(2)$ structure in the sense of \cite{ContiSalamon} is discovered. The interplay with CR-equations establishes what could be called a twistor space. Every non-constant sectional curvature metric implies the particular existence of four 1-forms
\begin{equation}
 \rho,\qquad \rho_1,\qquad \rho_2,\qquad \rho_3,
\end{equation}
closely related to the Ricci tensor, which develop into a new set of questions.

We start by recalling the differential geometry of $\cals$ in any dimension in order to establish an original structure Theorem. Next we present the general theory of a fundamental differential system of Riemannian geometry, introduced in \cite{Alb2011arxiv}, which is required but not essential. The Section following is devoted to the case of Riemannian manifolds of dimension 3 and it brings a plethora of new intrinsic objects and their related questions. In the last Section, again, useful general results are established, even though its main achievement is a new proof of \eqref{introd_derivadasdastres2formas}.


As one may care to notice, many classical textbooks on Differential Geometry contain a section which is devoted to the {intrinsic Riemannian geometry of \textit{surfaces}}. Frequently such a section ends the book. The theory develops much further after the works of Cartan and his celebrated structural equations, which are Lie algebra valued but fundamentally remain 1-forms. We thus hope our approach inspires the reader to the study of the 3d theory, supported on the natural differential system of fundamental 2-forms on the tangent sphere bundle as a new perspective on the intrinsic Riemannian geometry of 3-manifolds.

The author acknowledges Prof. Luigi Rodino, Universit\`a di Torino, for some helpful suggestions.

\subsection{The tangent manifold, orientation, metric and structure group} 

Let $M$ denote any oriented $n+1$-dimensional smooth manifold. The total space of $TM$, denoted $T_M$, is well-known to be a manifold of dimension $2n+2$, with a differentiable structure arising from the vector bundle structure associated to the manifold $M$ through the fibration $\pi:TM\lrr M$, i.e. the local trivializations as Cartesian products of neighbourhoods of $M$ with the vector space $\R^{n+1}$ and their transition maps linear within the fibres.

When $M$ is endowed with a linear connection $\na:\Gamma(M;TM)\lrr\Gamma(M;T^*M\otimes TM)$, there exists a canonical decomposition of $T{T_M}$ as $TT_M=H\oplus V$. Let us recall its definition and main properties in very quick steps. We have the vertical distribution $V=\ker\dx\pi\simeq\pi^\star TM$, a natural isomorphism, and then the horizontal distribution $H$ depending on $\na$. Clearly, $H$ is also isomorphic to $\pi^*TM$ through the map $\dx\pi$. Then we may define an endomorphism, indeed a tensor, transforming horizontal into vertical directions, via $\dx\pi$, and vanishing on verticals. It is enough to see it with lifts ($y\in TM$):
\begin{equation}\label{Bendomorphism}
 B:T{T_M}\lrr T{T_M} \ \qquad By^h=y^v,\quad By^v=0 .
\end{equation}
The manifold $T_M$ has two canonical vector fields. Namely, the tautological vertical vector field $\xi$, defined as $\xi_u=u,\ \forall u\in T_M$, and its mirror on the horizontal distribution, formally $B^{\mathrm{t}}\xi$, known as the geodesic-spray (\cite{Sakai}). The term \textit{mirror} means the image through $B$ is $\xi$. Indices $\cdot^h$ and $\cdot^v$ refer to the obvious canonical projections.
We have that $H=\ker(\pi^\star\na_\cdot\xi)$ and, $\forall y\in TT_M$,
\begin{equation}\label{nablaxi}
 \pi^\star\na_y\xi=y^v .
\end{equation}

$T_M$ inherits a linear connection, denoted $\na^*$ or $\na^\star$, preserving the cano\-ni\-cal decomposition:
\begin{equation}\label{nablasterisco}
  \na^*=\na^\star=\pi^*\na\oplus\pi^\star\na .
\end{equation}
The mirror endomorphism $B$ is parallel for such $\na^*$ by construction. The torsion of $\na^*$ is given by $\pi^*T^\na(v,w)\oplus\calri(v,w)$, $\forall v,w\in TT_M$, where
the vertical part is $\calri(v,w)=R^{\pi^\star\na}(v,w)\xi=\pi^\star R^\na(v,w)\xi$.

Any given frame in $H$ followed by its mirror in $V$ clearly determines a unique orientation on the manifold $T_M$. We convention the order `first $H$, then $V$', which is an issue when $\dim M$ is odd.

Let us now assume the $n+1$-dimen\-si\-o\-nal manifold $M$ is also Riemannian. Then we may consider the Sasaki metric on $T_M$ and quite immediately conclude the manifolds $T_M$ and $T_M\backslash$(zero section) have structural group $\SO(k)\times\SO(k)$ where, respectively, $k=n+1$ and $n$. Also we assume $\na$ is a metric connection, $\na\langle\cdot,\cdot\rangle=0$. Then the larger structure is of course parallel for $\na^*$, whereas the smaller never is because $\xi$ is not parallel. The vector bundle isomorphism $B_|:H\rr V$, always parallel, becomes a metric-preserving map.

Finally  we recall the map $J=B-B^{\mathrm{t}}$ gives the Sasaki almost complex structure. Further, the $\mathrm{GL}(n+1,\C)$ structure on $T_M$ is integrable if and only if $T^\na=0$ and $R^\na=0$ (cf. \cite{Alb2008} and the references therein).

\subsection{The tangent sphere bundle}

Let $\na$ be the Levi-Civita connection and consider the constant radius $s$ tangent sphere bundle
\begin{equation}
 S_sM:=\{u\in TM:\ \|u\|=s\}\lrr M .
\end{equation}
Except for the Introduction section, we shall consider any radius $s$ tangent sphere bundle.

Let $\cals=\cals_{s,M}$ denote the total space of $S_sM$. Differentiating $\langle\xi,\xi\rangle=s^2$, we find that $T{\cals}=\xi^\perp$ and hence that ${\cals}$ is always orientable --- the restriction of $\xi$ being an \textit{outward} normal. By the Gram-Schmidt process and the orthogonal group action, for any $u\in\cals$ we may find a local vertical and a horizontal orthonormal frame $e_0,e_1,\ldots,e_n$ on a neighbourhood of $u$ in ${\cals}$ such that $e_0=\frac{1}{s}B^{\mathrm{t}}\xi_u\in H_u$.

With the dual horizontal co-framing, the identity $\pi^*\vol_{_M}=e^{01\cdots n}$ follows\footnote{We adopt the usual notation $e^a\wedge e^b\wedge\cdots\wedge e^c=e^{ab\cdots c}$. Also we shall use $\llbracket...\rrbracket$ later-on to denote the linear $\R$-span of that whatever appears between the brackets.}. Joining in the vertical 1-forms $\xi^\flat,e^{n+1},\ldots,e^{2n}$, such that
\begin{equation}
e^{n+i}\circ B=e^i,\ \forall 1\leq i\leq n
\end{equation}
and giving also a dual frame satisfying $e^{n+i}(e_j)=e^i(e_{j+n})=0$,
$e^{n+i}(e_{j+n})=e^i(e_j)=\delta_j^i,\ \forall i,j$, we find the volume form of $T_M$:
\begin{equation}
 e^{012\cdots n}\wedge\frac{1}{s}\xi^\flat\wedge
e^{(n+1)\cdots(2n)} = (-1)^{n+1}\frac{1}{s}\,\xi^\flat\wedge\vol\wedge\alpha_n\ .
\end{equation}
We use $\vol:=\pi^*\vol_{_M}$ and the $n$-form $\alpha_n$ on $T_M$ which is
defined as the interior product of $\xi/s$ with the vertical lift of the volume
form of $M$. Hence, cho\-osing appropriate $\pm\xi$ as outward normal direction, the
cano\-ni\-cal orientation of the Riemannian submanifold $\cals$, i.e. with the induced metric, agrees with $\vol\wedge\alpha_n=e^{01\cdots(2n)}$. The direct orthonormal frame $e_0,e_1,\ldots,e_n,\frac{1}{s}\xi,e_{n+1},\ldots,e_{2n}$ is said to be \textit{adapted}. Without farther referring the principal bundle of adapted frames, we summarize the unique structure of $T_M$ as follows.
\begin{teo}\label{structuregroupofTM:theorem}
 The tangent manifold $T_M$ has structural group $\SO(n+1)$, through the diagonal action on $\R^{n+1}\oplus\R^{n+1}$, and the induced connection $\na^*$ is reducible. 
 
 The submanifolds $T_M\backslash$(zero section) and $\cals$ have structural group $\SO(n)$. The induced connection $\na^*$, in the first case, and its canonical $T\cals$-valued correction, in the second, are both not reducible.
\end{teo}
\vspace{1mm}
\noindent\textsc{Remark.}
The reduction of the structural group of $S_sM$ from $\SO(2n+1)$ to the middle subgroup $\mathrm{U}(n)\supseteq\SO(n)$ induces an \textit{integrable} structure on $\cals$ under certain conditions. Namely, the total space $\cals$ is a Sasakian manifold if and only if the base $M$ has constant sectional curvature $\frac{1}{s^2}$ (\cite{Alb2011arxiv,DrutaRomaniucOproiu}). In a heuristic interpretation, this may be seen as follows. The quotient distribution $T\cals/\R e_0$ agrees infinitesimally with the tangent space to a quotient K\"ahler if and only if the horizontal $n$-plane $H\cap e_0^\perp$ is tangent to a submanifold which agrees infinitesimally, through $J$, with the sphere $S^n_s$, lying perpendicularly, as the standard fibre of $S_sM$. The integrability result says that $M$ must be locally a sphere $S^{n+1}_s$. The simply connected case is thus the Stiefel manifold $V_{n+2,2}$.

\subsection{Recalling the fundamental differential system}

We denote by $\theta$ the 1-form on $\cals$ defined as
\begin{equation}\label{mu}
\theta=\langle\xi,B\,\cdot\,\rangle=s\,e^0.
\end{equation}
It is well-known that $\theta$ defines a metric contact structure on $\cals$. With our coordinate-free instruments we immediately find the known result $\dx\theta=e^{(1+n)1}+\cdots+e^{(2n)n}$, cf. Section \ref{sec:Proofs}. Using $\na^*$ one also computes directly $\dx\theta(v,w)=\langle v,Bw\rangle-\langle w,Bv\rangle$, $\forall v,w\in T\cals$, confirming the independence of $s$.

After the above definitions and necessary digression on the geometry of $\cals$, we are ready to recall the natural global $n$-forms $\alpha_0,\alpha_1,\ldots,\alpha_n$ associated to the given oriented Riemannian manifold. Together, $\theta$ and the $\alpha_i$ form the exterior differential system discovered in \cite{Alb2011arxiv}.

We first write $\pi^\star\vol_{_M}$ for the vertical lift of the volume form of $M$ (this is not the pull-back form; always a $\pi^\star$ shall denote a vertical lift). Recall we already have
\begin{equation}\label{alphazero}
 \alpha_n = \frac{1}{s}\,\xi\lrcorner({\pi}^\star\vol_{_M}) .
\end{equation}
Now, for each $0\leq i\leq n$ we define, $\forall v_1,\ldots,v_n\in T\cals$,
\begin{equation}\label{alpha_itravez}
 \alpha_i(v_1,\ldots,v_n) = \frac{1}{i!(n-i)!}\sum_{\sigma\in\mathrm{Sym}(n)}\mathrm{sg}(\sigma)\,  \alpha_n(Bv_{\sigma_1},\ldots,B v_{\sigma_{n-i}},v_{\sigma_{n-i+1}},\ldots,v_{\sigma_n}).
\end{equation}
For convenience one also defines $\alpha_{-1}=\alpha_{n+1}=0$. 

By uniqueness of the Levi-Civita connection, $\na$ is invariant for every isometry of $M$ and hence all the $\alpha_i$ are invariant by isometry.

We shall use the notation
\begin{equation}
 R_{lkij} = \langle R^\na(e_i,e_j)e_k,e_l\rangle= 
 \langle \na_{e_i}\na_{e_j}e_k-\na_{e_j}\na_{e_i}e_k-\na_{[e_i,e_j]}e_k,e_l\rangle.
\end{equation}
\begin{teo}[1st-order structure equations, \cite{Alb2011arxiv}] \label{derivadasdasnforms}
We have
\begin{equation}\label{dalphai}
 \dx\alpha_i=\frac{1}{s^2}(i+1)\,\theta\wedge\alpha_{i+1}+\calri\alpha_i
\end{equation}
where
\begin{equation}\label{Ralphai}
 \calri\alpha_i = \sum_{0\leq j<q\leq n}\sum_{p=1}^nsR_{p0jq}\,e^{jq}\wedge
e_{p+n}\lrcorner\alpha_i.
\end{equation}
\end{teo}
Defining $r=\frac{1}{s^2}\pi^\star\ric(\xi,\xi)=\sum_{j=1}^nR_{j0j0}$, a smooth function on $\cals$ determined by the Ricci curvature of $M$, a few computations in \cite{Alb2011arxiv} show that $\calri\alpha_0=0$ and $\calri\alpha_{1}=-r\,\theta\wedge\alpha_0$. This is
\begin{equation}\label{dalphanen-1}
 \dx\alpha_0=\frac{1}{s^2}\,\theta\wedge\alpha_{1},\qquad\quad
 \dx\alpha_{1}=\frac{2}{s^2}\,\theta\wedge\alpha_{2}-sr\,\vol.
\end{equation}
Moreover, the differential forms \,$\theta$,\ $\alpha_n$ and $\alpha_{n-1}$ are always coclosed (cf. \cite[Proposition 2.3]{Alb2011arxiv} or \eqref{Basicstrutequations2} below). In every degree we have $\alpha_i\wedge\dx\theta=0$, and hence
\begin{equation}\label{drxialphai}
 \dx(\calri\alpha_i)=\frac{1}{s^2}(i+1)\,\theta\wedge\calri\alpha_{i+1},\qquad\quad
 \dx\theta\wedge\calri\alpha_i=0.
\end{equation}

We remark once again there are no further assumptions on $M$. It is just an oriented $n+1$-dimensio\-nal Riemannian manifold, from which the associated fundamental exterior differential system is defined as the ideal $\{\theta,\alpha_0,\ldots,\alpha_n\}\Omega^*_\cals$.

\subsection{Applications to special Riemannian structures}

The author has developed in \cite{Alb2011arxiv} some applications of the differential system. One missing detail is a simple verification of the formulae, e.g. through charts, for the case $n=1$. In \cite{SingerThorpe} we find a proof of this already non-trivial case. One of the purposes of this article is to give a new further enlightening proof of equations (\ref{dalphai},\ref{Ralphai}). We obtain it, in a quite independent Section \ref{sec:Proofs}, for the cases $n=1$ and $n=2$.

In the study of the differential system we are challenged with finding the associated calibration $p$-forms. We shall require the following forms, often giving rise to invariant calibrations. For example, we recall that in \cite{Alb2010,Alb2013,AlbSal2009,AlbSal2010} a $\gdois$ structure is found on $\cals$ for any oriented 4-manifold $M$; which are cocalibrated if and only if $M$ is Einstein.

For any $0\leq i\leq n$ we have:
\begin{equation}\label{Basicstrutequations2}
\begin{split}
*\,(\dx\theta)^i&= \frac{(-1)^{\frac{n(n+1)}{2}}i!}{s(n-i)!}\, \theta\wedge(\dx\theta)^{n-i} \\
*\alpha_i&=\frac{(-1)^{n-i}}{s}\,\theta\wedge\alpha_{n-i} .
\end{split}
\end{equation}
It is important to have in mind that $\alpha_i\wedge\dx\theta=0$ and $\alpha_i\wedge\alpha_j=0,\ \forall j\neq n-i$. The Hodge star-operator $*$ on $\cals$ satisfies $**=1_{\Lambda^*}$.

We recall a first result involving $\alpha_{n-2}$ and a 1-form playing a central role: $\rho=\frac{1}{s}\,\xi\lrcorner\pi^\star\ric$. It is thus defined through the vertical lift of the Ricci tensor and by restricting to $\cals$. With an adapted frame, we deduce
\begin{equation}\label{defofrho}
 \rho=\sum_{a,b=1}^nR_{a0ab}\,e^{b+n} .
\end{equation}
One sees that expressions such as \eqref{defofrho} are independent of the choice of adapted frame. The tangent vector $s\,e_0$ is the horizontal tautological lift of the point $u\in\cals$ in question, just as $\xi$ is the vertical. Recall also these two vectors are fixed in the adapted frame. 

Let us denote the co-differential by $\delta=-*\dx*$.
\begin{teo}[\cite{Alb2011arxiv}]\label{metricaEinsteinealpha2}
In any dimension we have $\dx*\alpha_{n-2}=\rho\wedge\theta\wedge\alpha_0$. Henceforth, the metric on $M$ is Einstein if and only if $\delta\alpha_{n-2}=0$.
\end{teo}
We have also, quite easily,
\begin{equation}\label{derivadade_remgeral}
 \dx r=\sum_{i=0}^n(\na_i\ric_{00})e^i+\frac{2}{s}\rho.
\end{equation}
We shall prove a further identity in Section \ref{subsec:GeneralComputations}, Proposition \ref{prop:derivadade_rhogeral}:
\begin{equation}\label{derivadade_rhogeral}
  \dx\rho=\frac{1}{s}\sum_{i=0}^ne^i\wedge\xi\lrcorner\na^\star_i\ric .
\end{equation}

Let us recall the interesting case of constant sectional curvature $c$. Since the Riemann curvature tensor is $R_{qpij}=c(\delta_{iq}\delta_{jp}-\delta_{ip}\delta_{jq})$, we find $\calri\alpha_i=-c(n-i+1)\,\theta\wedge\alpha_{i-1}$. In particular all the $*\alpha_i$ are closed $n+1$-forms.

\section{The 3-dimensional differential system}

We now consider any oriented Riemannian 3-manifold $M$, together with the 5-di\-mensional Riemannian manifold $\cals$ given by the tangent sphere bundle $S_{M,s}\lrr M$ equipped with Sasaki metric and canonical orientation.

\subsection{Representation spaces}

On $\cals$ we have the contact 1-form, $\theta=s\,e^0$, which is clearly invariant for the action of $\SO(2)$ on $\R^{1+2+2}$, cf. Theorem \ref{structuregroupofTM:theorem}. This is the trivial action on the 1-dimensional summand and the diagonal action on the orthogonal complement.

From the definition we find the four \textit{global} invariant 2-forms, frame choice independent,
\begin{equation}\label{thefourinvariants}
 \alpha_0=e^{12},\qquad \alpha_1=e^{14}-e^{23},\qquad \alpha_2=e^{34},\qquad\dx\theta=e^{31}+e^{42}.
\end{equation}
We also see
\begin{equation} 
  \alpha_0\wedge\alpha_1=\alpha_2\wedge\alpha_1=\alpha_i\wedge\dx\theta=0,\:\ \forall i=0,1,2,
\end{equation}
\begin{equation}
   \frac{1}{s}\,*\theta=\alpha_0\wedge\alpha_2=-\frac{1}{2}\,\alpha_1\wedge\alpha_1=-\frac{1}{2}\,(\dx\theta)^2
\end{equation}
and
\begin{equation}
 *\dx\theta=-\frac{1}{s}\,\theta\wedge\dx\theta,\quad 
 *\alpha_0=\frac{1}{s}\,\theta\wedge\alpha_2,\quad
 *\alpha_1=-\frac{1}{s}\,\theta\wedge\alpha_1,\quad
 *\alpha_2=\frac{1}{s}\,\theta\wedge\alpha_0.
\end{equation}
\begin{prop}\label{prop_decompwedgetwo}
 The representation under $\SO(2)$ above, induced on the vector bundle $\Lambda^2T^*\cals$, corresponds with the decomposition
\begin{equation}\label{decompwedgetwo}
 \Lambda^2\R^5=4\R^1\oplus W_1\oplus W_2\oplus W_3
\end{equation}
where we have the four 1-dimensional invariants from \eqref{thefourinvariants} and three irreducible orthogonal subspaces $W_i$ defined by
\begin{equation}
 W_1=\llbracket e^{01},e^{02}\rrbracket, \qquad W_2=\llbracket e^{03},e^{04}\rrbracket
\end{equation}
and
\begin{equation}
 W_3=\llbracket f_1,f_2\rrbracket
\end{equation}
where
\begin{equation}\label{The_f_forms}
 f_1:=e^{14}+e^{23},\qquad f_2:=e^{31}-e^{42}.
\end{equation}
\end{prop}
(Notice the 2-forms $f_1,f_2$ are not invariantly defined.)

It is trivial to write the $W_i$ as eigenspaces of certain endomorphisms. On the 4-dimensional side $e_0^\perp$, we note $W_3$ is composed of $*_4$-selfdual forms. In particular, it is orthogonal to the $\alpha_i$ and $\dx\theta$. Using the Hodge isomorphism, we deduce the decomposition of $\Lambda^3\R^5$ into irreducibles. $\Lambda^1\R^{5}$ is an elementary case and $\Lambda^4\R^5=\llbracket*\theta\rrbracket\oplus W_1\alpha_2\oplus W_2\alpha_0$. Since the canonical epimorphism $\Lambda^1\R^5\otimes\Lambda^2\R^5\lrr\Lambda^3\R^5$ has a kernel of dimension 40, there are many equivalent representations in the space of 3-forms arising from \eqref{decompwedgetwo}.

Finally, the 1-form defined in \eqref{defofrho} is an irreducible:
\begin{equation}
 \rho=\frac{1}{s}\,\xi\lrcorner\pi^\star\ric=R_{1012}e^4-R_{2012}e^3.
\end{equation}
Recalling the scalar function $r=\frac{1}{s^2}\ric(\xi,\xi)=R_{1010}+R_{2020}$, we find that we may write it using scalar and sectional curvatures as $r=\frac{1}{2}\scal-\mathrm{c}(\{e_1,e_2\})=\frac{1}{2}\scal-R_{1212}$. Clearly, $\ric=\lambda\langle\cdot,\cdot\rangle$, for some constant $\lambda$, implies $M$ has constant sectional curvature $\lambda/2$.

\subsection{Natural $\SU(2)$ structures or \textit{twistor} space}

There exists an almost complex structure on each sub-vector bundle $H_0=H\cap e_0^\perp=\llbracket e_1,e_2\rrbracket$ and $V_0=V\cap \xi^\perp=\llbracket e_3,e_4\rrbracket$ of $T\cals$. We shall denote by $I_+$ and $I_-$ the maps defined, according to $\pm$, by
\begin{equation}\label{segundaestruturacomplexa}
  e_0\mapsto 0,\qquad e_1\mapsto e_2 \mapsto-e_1,\qquad
   e_3\mapsto \pm e_4 \mapsto -e_3.
\end{equation}
$I_+,I_-$ are invariantly defined commuting endomorphisms of $T\cals$. We choose $I_+$ to induce complex structures, by restriction, on the vector bundles $H_0$ and $V_0$. This choice preserves orientation in the sense that on $H_0\oplus V_0$ we have
$JI_+J^t=JI_+\inv{J}=I_+$. Notice on the other hand that $J$ and $I_-$ anti-commute, giving an immediate proof of the following result.
\begin{teo}
 For every oriented Riemannian 3-manifold $M$ the 5-dimensional Riemannian contact manifold $\cals$ admits an $\SU(2)$ structure in the sense of Conti-Salamon, defined by $(\theta,J,I_-)$.
\end{teo}
(This structure on $\cals$ now truly recalls us of the twistor space of 4-manifolds and one certainly could ponder the use of the same term.)

Regarding natural integrability questions, they shall be studied elsewhere since there are many conditions to be verified within the classification of $\SU(2)$ structures. Indeed, using weight coefficients, the present structure admits several variations which yield new hypo, double-hypo or Sasaki-Einstein manifolds. Particularly interesting are the hypo manifolds, i.e. the real 5-dimensional hypersurfaces of an $\SU(3)$ manifold with the induced Conti-Salamon $\SU(2)$ structure. The theory was started in \cite{ContiSalamon} and developed in \cite{deAndFerFinoUgar,BedulliVezzoni,FIMU}, just to mention a few.

The complex line bundles $H_0$ and $V_0$ are very particular to dimension 3, due to $\SO(2)=\uni(1)$. The vertical $\C$-line bundle $V_0$ is clearly the holomorphic tangent bundle when restricted to each $S^2$ fibre, with $\alpha_2$ restricting to the K\"ahler class even if $\dx\alpha_2\neq0$ globally.

Like $\rho$ above, we have global 1-forms defined by
\begin{equation}\label{osrhostodos}
 \begin{split}
 \rho=R_{1012}e^4-R_{2012}e^3 , & \\ \rho_1=\rho B=R_{1012}e^2-R_{2012}e^1 , & \\
  \rho_2=\rho I_+B=R_{1012}e^1+R_{2012}e^2 , & \\
  \rho_3=\rho I_+=R_{1012}e^3+R_{2012}e^4 . & 
 \end{split}
\end{equation}
The formulae $*(\rho\wedge\vol)=\rho_3$ and $*(\rho_3\wedge\vol)=-\rho$ are helpful. As well as the following prove to be: $*\rho_1=\frac{1}{s}\,\theta\wedge\rho_2\wedge\alpha_2$ and $*\rho_2=-\frac{1}{s}\,\theta\wedge\rho_1\wedge\alpha_2$.

From the existence of equivalent representations in $\Lambda^3$ we obtain the next result.
\begin{prop}\label{prop_equividentities}
The following identities hold:
\begin{equation}\label{rhorelations}
  \begin{split}
  \rho\wedge\alpha_0=-\rho_1\wedge\alpha_1=-\rho_2\wedge\dx\theta , &\\ \rho_1\wedge\alpha_2=\rho_3\wedge\dx\theta=-\rho\wedge\alpha_1 , & \\
  \rho_2\wedge\alpha_1=-\rho_3\wedge\alpha_0=-\rho_1\wedge\dx\theta , & \\ \rho_3\wedge\alpha_1=\rho\wedge\dx\theta=-\rho_2\wedge\alpha_2.  &
 \end{split}
\end{equation}
\end{prop}

\subsection{Exterior derivatives}

From the general formula in \eqref{dalphanen-1} we have $r=R_{1010}+R_{2020}$ and
\begin{eqnarray}
& & \dx\alpha_0=\frac{1}{s^2}\,\theta\wedge\alpha_1,  \label{derivadasdastres2formas_alpha2} \\
& & \dx\alpha_1=\frac{2}{s^2}\,\theta\wedge\alpha_2-r\,\theta\wedge\alpha_0 .
\label{derivadasdastres2formas_alpha1}
\end{eqnarray}
The first derivatives are already decomposed into irreducibles, cf. Proposition \ref{prop_decompwedgetwo}.

Given any contact $2n+1$-manifold, such as $(\cals,\theta)$, it is known that the $\dx$-closed ideal $\cali$ generated by $\theta$ contains the whole exterior algebra above the degree $n$, cf. \cite[Theorem 1.1]{BGG}. Each $n+1$-form $\Pi$ may thus be written globally as a form in $\cali$. Moreover, any class $[\Pi]$ has a unique representative congruent with $0\mod\theta$ in the differential cohomology $H^{n+1}(\cali)$. Such unique representative of $\dx\Lambda$, given any $n$-form or Lagrangian $\Lambda$ on the contact manifold, is called the Poincar\'e-Cartan form of $\Lambda$. It is very important in the study of the Euler-Lagrange system $\{\theta,\dx\theta,\Lambda\}$, specially for the associated variational principle, cf. \cite{BGG}.

We thus have immediately the Poincar\'e-Cartan forms of $\alpha_0$ and $\alpha_1$.
\begin{teo}\label{teorema_dalpha0compostinho}
 The decomposition of $\dx\alpha_2$ is given by
\begin{equation}\label{dalpha0composto}
 \dx\alpha_2 = \theta\wedge\gamma-\frac{r}{2}\,\theta\wedge\alpha_1+s\,\alpha_0\wedge\rho\qquad \in\quad *W_3\oplus\llbracket*\alpha_1\rrbracket\oplus*W_2
\end{equation}
where, by \eqref{The_f_forms}, the 2-form $\gamma$ is defined as
\begin{equation}
 \gamma:=R_{1002}f_2+\frac{1}{2}(R_{1001}-R_{2002})f_1\ \ \in\  W_3.
\end{equation}
The Poincar\'e-Cartan form of $\alpha_2$ is
\begin{equation}
 \Pi=\theta\wedge\bigl(\gamma-\frac{r}{2}\alpha_1-s\,\dx\rho_2\bigr).
\end{equation}
\end{teo}
\begin{proof}
 The reader may easily see the 2-form $\gamma$ is independent of the choice of the orthonormal frame $e_1,e_2$ such that $e_0,e_1,e_2$ is positively oriented. From the coincidences in Proposition \ref{prop_equividentities}, we shall need  $\alpha_0\wedge\rho=-\rho_2\wedge\dx\theta$. Starting from Theorem \ref{derivadasdasnforms}, also cf. \eqref{derivadada2formalpha0}, we obtain
 \begin{eqnarray}
 \dx\alpha_2 & =& s\,e^0\bigl(R_{1001}e^{14}-R_{2002}e^{23}+R_{1002}e^{24}-R_{2001}e^{13}\bigr)+ \nonumber \\
  & &\hspace{33mm} + sR_{1012}e^{124}-sR_{2012}e^{123}   \label{dalpha0porextenso} \\
  &=& \theta\wedge\gamma-\frac{r}{2}\,\theta\wedge\alpha_1
           +s\,\alpha_0\wedge\rho        \nonumber    \\
  &=& \theta\wedge(\gamma-\frac{r}{2}\,\alpha_1)-s\,\rho_2\wedge\dx\theta  \nonumber  \\   &=& \theta\wedge(\gamma-\frac{r}{2}\,\alpha_1-s\,\dx\rho_2)+  
  \dx(s\,\rho_2\wedge\theta) .\nonumber
 \end{eqnarray}
 The first part of the result is thus immediate after the first line in the computation above, since the representation subspaces are known. For the second part, the Poincar\'e-Cartan form $\Pi$ of $\alpha_2$ is finally $\Pi=\dx(\alpha_2-s\,\rho_2\wedge\theta)$. 
\end{proof}
We recall that a new proof of the main derivatives is given in the last Section. In dimension 3 there is no place for the Weyl curvature tensor (\cite{Hamil}); It is trivial to see that an Einstein metric is in fact of constant sectional curvature.
\begin{coro}
The following assertions are equivalent on a connected 3-manifold: $M$ has constant sectional curvature; $r$ is constant; $\rho=0$; $\gamma=0$; $\dx\alpha_2=-\frac{r}{2}\,\theta\wedge\alpha_1$.
\end{coro}
\begin{proof}[Proof (before the Theorem of Schur)]
 The only implication which offers some doubt is that the last statement implies the first. So differentiating $\dx\alpha_2$ again we obtain easily from $\alpha_i\wedge\dx\theta=0$ and \eqref{derivadasdastres2formas_alpha1} that $\dx r\wedge\theta\wedge\alpha_1=0$.
 Since $\dx r=\sum_{l=0}^4\dx r(e_l)e^l$, it is easy to see $\dx r=\dx r(e_0)e^0$. Hence $r$ does not vary on vertical directions, nor on any horizontal as it must then be concluded (if one prefers, every closed 1-form $f\theta$ must vanish).\footnote{By the well-known Theorem of Schur, we knew already that $r$ is a constant. This classical result is proved in any dimension $\geq3$ in \cite[Proposition 3.1]{Alb2011arxiv} with the new system. The reader may be defied by the 2-dimensional system \eqref{dalphais_m=2}, knowing that sectional curvature is constant in general only on each $S^1$ fibre.}
\end{proof}
A relevant equation comes from differentiating \eqref{derivadasdastres2formas_alpha1} again, cf. \eqref{derivadade_remgeral}:
\begin{equation}
 (2\rho-s\,\dx r)\wedge\theta\wedge\alpha_0=0.
\end{equation}

\subsection{$\na\ric$ and the co-differentials}

In differentiating \eqref{dalpha0composto} we are confronted with the exterior derivatives of $\gamma$ and $\rho$. While the former is a new mysterious object, the interpretation of the latter is more accessible. We have from \eqref{derivadade_rhogeral} that $\dx\rho=\sum_{i=0,\:j=1}^2(\na_i\ric)_{0j}e^{i,j+2}$, and hence we may apply the representation of 2-forms ($F_1,F_4\in C^\infty_\cals$):
\begin{equation}\label{dxrho_classif}
 \dx\rho=F_1\alpha_1+F_2+F_3+F_4\dx\theta\quad\in\quad \llbracket\alpha_1\rrbracket\oplus W_2\oplus W_3\oplus\llbracket\dx\theta\rrbracket .
\end{equation}
These forms suggest a classification of the 3-tensor $\na\ric\in\Gamma(M;T^*M\otimes S^2T^*M)$ in parallel with that found by Gray in \cite{Gray} in general. Through the geometry of $\cals$ over the 3-dimensional base, we obtain 16 different cases which do not repeat the 8 representation classes under $\SO(3)$. It is a new description, one might agree. The following conditions are invariant of the orthonormal base of $e_0^\perp\subset TM$ for each $e_0$. We say the metric is:

- Ricci type I if $(\na_1\ric)_{02}=(\na_2\ric)_{01}$. Equivalently, $F_1=0$.

- Ricci type II if $(\na_0\ric)_{01}=(\na_0\ric)_{02}=0$. Equivalently, $F_2=0$.

- Ricci type III if $(\na_1\ric)_{02}=-(\na_2\ric)_{01}=0$. Equivalently, $F_3=0$.

- Ricci type IV if $(\na_1\ric)_{10}+(\na_2\ric)_{20}=0$.  Equivalently, $F_4=0$.

Notice Ricci type III is included in I and II is included in IV, due to symmetries. III is also equivalent to $(\na_1\ric)_{01}=(\na_2\ric)_{02}$. Also note the uniqueness of such a decomposition is not assured, although of course each lies in a minimal $\SO(3)$ representation space. This classification has a different meaning from that of representation theory of the base manifold structure group. In short terms, the condition in each type means that the equations must be satisfied $\forall m\in M$, $\forall e_0\in T_mM$ and \textit{one} orthonormal basis $e_1,e_2$ of $e_0^\perp$.

None of the above Ricci types seem to imply constant scalar curvature (CSC). Following the results on Einstein-like manifolds, cf. \cite{Besse,Gray}, we have in general from the second Bianchi identity and an orthonormal basis:
\begin{equation}
 \sum_i\na_u\ric(e_i,e_i)=2\sum_{i}\na_{e_i}\ric(e_i,u),\quad \forall u\in TM.
\end{equation}
CSC is the same as the vanishing of the left hand side. Such space is composed of two $\SO(3)$-irreducibles, the well-known Codazzi and Killing type tensors. The orthogonal to CSC Ricci type is not of any specific type I to IV.

On the other hand, a particularly interesting type of metric are those which satisfy the \textit{recurrent} condition on the Ricci tensor: $\na\ric=\omega\otimes\langle\ ,\ \rangle$ for some 1-form $\omega$ on $M$. This is an $\SO(3)$-reducible which clearly belongs to all four Ricci types above.

The following table details some further coincidences, easy to check.
\begin{teo}
We have:
\begin{center}
\begin{tabular}{cc}
If $\dx\rho\,\in$   & then Ricci type  \\
\hline 
 $\llbracket\alpha_1\rrbracket\oplus W_2\oplus W_3\oplus\llbracket\dx\theta\rrbracket$ &   \\
 $W_2\oplus W_3\oplus\llbracket\dx\theta\rrbracket$ & I\\
 $\llbracket\alpha_1\rrbracket\oplus W_3\oplus\llbracket\dx\theta\rrbracket$ & II\\
 $\llbracket\alpha_1\rrbracket\oplus W_2\oplus\llbracket\dx\theta\rrbracket$& III \\
 $\llbracket\alpha_1\rrbracket\oplus W_2\oplus W_3$ & IV \\
 $W_2\oplus W_3$ & I and IV \\
 $W_3\oplus\llbracket\dx\theta\rrbracket$ & I and II\\
 $W_2\oplus\llbracket\dx\theta\rrbracket$ & III \\
 $\llbracket\alpha_1\rrbracket\oplus\llbracket\dx\theta\rrbracket$ & II and III \\
 $\llbracket\alpha_1\rrbracket\oplus W_3$ & II \\
 $\llbracket\alpha_1\rrbracket\oplus W_2$ & II and III \\ 
 $W_3$ & I and II \\
 $\{0\}\cup\llbracket\dx\theta\rrbracket\cup W_2\cup \llbracket\alpha_1\rrbracket$ & II and III 
\end{tabular}
\end{center}
\end{teo}
Having the derivative of $\rho$, we pass to another kind of questions.
\begin{prop}
The following identities hold:
\begin{equation}
  \begin{split}
&  \qquad\qquad\quad\delta\dx\theta=-\frac{1}{s^2}\,\theta, \ \qquad\delta\theta=0,\\
&\qquad \delta\alpha_0=-s\,\rho_3,\quad\quad \delta\alpha_1=0,\quad\quad \delta\alpha_2=0,\\
&\delta\rho=?,\qquad \delta\rho_1=2F_4,\qquad\delta\rho_2=2F_1,\qquad\delta\rho_3=0.
\end{split}
\end{equation}
\end{prop}
\begin{proof}
 This is a simple exercise which requires several identities deduced earlier. For instance, $\delta\rho_1=-*\dx*\rho_1=-\frac{1}{s}*\dx(\theta\wedge\rho_2\wedge\alpha_2)=\frac{1}{s}*\,\dx(\theta\wedge\rho\wedge\dx\theta)=-\frac{1}{s}*(\theta\wedge F_4\dx\theta\wedge\dx\theta)=2F_4$.
\end{proof}
Moreover, $\alpha_1$ is co-exact. The Hodge decomposition of $\alpha_0$ and $\alpha_2$ is unknown to the author. $\dx\theta$ is always an eigenform of the Laplacian $\Delta=\dx\delta+\delta\dx$. In praise of this operator we write the following result (giving more three eigenforms).
\begin{prop}
 Let $M$ have constant sectional curvature $c$. Then
 \begin{equation}
  \Delta \alpha_0=\frac{2}{s^2}\,\alpha_0-2c\,\alpha_2,\qquad 
  \Delta \alpha_1=\frac{2+2c^2s^4}{s^2}\,\alpha_1,\qquad 
  \Delta \alpha_2=-2c\,\alpha_0+2c^2s^2\,\alpha_2.
 \end{equation}
\end{prop}

\subsection{Integration along the fibre}

Besides the fibre-constant $\scal$, we have other interesting scalar functions invariantly defined on $\cals$. Using any adapted frame $e_0,e_1,e_2$ (recall 0 stands for the horizontal replica of the unit direction of the point $u\in\cals$ in question), such functions are:
\begin{equation}\label{sectionalcurvature}
  c=R_{1212}   , 
\end{equation}
\begin{equation}\label{sectionalcurvaturer}
  r=R_{1010}+R_{2020}=\frac{1}{2}\scal-c  , 
\end{equation}
\begin{equation}\label{sectionalcurvaturerho}
   p^2=\|\rho\|^2={R_{1012}}^2+{R_{2012}}^2 
\end{equation}
and
\begin{equation}
 \begin{split}\label{sectionalcurvaturegamma}
   q^2=\|\gamma\|^2 &=2{R_{1002}}^2+\frac{1}{2}(R_{1001}-R_{2002})^2 \\
                    &=\frac{1}{2}r^2-2\det R_{\cdot00\cdot} .
\end{split}
\end{equation}
One also finds the relations $\rho_3\wedge\rho=p^2\,\alpha_2$ and $\rho_2\wedge\rho_1=p^2\,\alpha_0$ where $p=\|\rho\|$. We note the remaining four similar products are not irreducible. With $q=\|\gamma\|$, we may further write
\begin{equation}
 sp^4\,\vol_\cals=\theta\wedge\rho\wedge\rho_1\wedge\rho_2\wedge\rho_3,\qquad  
 q^2\,\alpha_0\wedge\alpha_2=\gamma\wedge\gamma.
\end{equation}
Recall that $\dx(\alpha_i\wedge\alpha_j)=0$ for all $i,j=0,1,2$, and so, in particular, we may take the integral over $\cals$ of the following 5-form in various ways:
\begin{equation}
 r\,\theta\wedge\alpha_0\wedge\alpha_2= -\dx\alpha_1\wedge\alpha_2=\alpha_1\wedge\dx\alpha_2 =-\frac{r}{2}\,\theta\wedge\alpha_1\wedge\alpha_1 .
\end{equation}

Integration along the fibre obtained for any form or real function $f\in C^0$ on $\cals$ is also interesting:
\begin{equation}\label{fibreintegrationfunctioningeneral}
 \check{f}(x)=\frac{1}{s^2}\int_{\inv{\pi}(x)}f\,\alpha_2 \qquad(x\in M).
\end{equation}
\begin{teo}
 With $\pi=3.14...$ and  the norm $\|R\|^2=\sum{R_{abcd}}^2$, we have:
 \begin{equation}\label{variasintegracoesaolongodafibra}
  \begin{split}
& \check{1}=4\pi, \qquad\qquad 
 \check{c}=\frac{2\pi}{3}\scal, \qquad\qquad
 \check{c^2}=\frac{\pi}{15}(2\|R\|^2+\scal^2),  \\
& \hspace{7mm} \check{r}=\frac{4\pi}{3}\scal, \qquad\qquad
 \check{r^2}=\frac{2\pi}{15}(\|R\|^2+6\,\scal^2),  \\
&  \check{p^2}=\frac{\pi}{15}(3\|R\|^2-2\,\scal^2), \qquad\qquad
 \check{q^2}=\frac{2\pi}{15}(3\|R\|^2-2\,\scal^2) .
  \end{split}
 \end{equation} 
\end{teo}
\begin{proof}
 The sum $\sum {R_{abcd}}^2$ runs over all indices of an orthonormal frame. The result is expected by Chern-Weyl theory, so we just give details of the common tools needed to solve the computations of \eqref{variasintegracoesaolongodafibra}, of increasing complexity. Notice that all the functions are independent of the orientation on the $S^2$-fibres and also of the length of the ray. In order to integrate them, we take any fixed frame $\mathbf{i},\mathbf{j},\mathbf{k}$ of $\R^3$, so in particular
 \[ \|R\|^2=4({R_{\mathbf{ijij}}}^2+{R_{\mathbf{ikik}}}^2+{R_{\mathbf{jkjk}}}^2)+ 8({R_{\mathbf{ijik}}}^2+{R_{\mathbf{ijkj}}}^2+{R_{\mathbf{ikjk}}}^2),  \]
 and the coordinates $0\leq\theta <2\pi,\ -1<z<1$ applied in $u=e_0=aw+z\mathbf{k}\in S^2\subset\R^3$, where $w=\mathbf{i}\cos\theta+\mathbf{j}\sin\theta$. Of course we assume $a>0,\ a^2+z^2=1$. With this choice, an adapted frame of $T_uS^2$ is given by $e_1=\widetilde{w}$ and $e_2=-zw+a\mathbf{k}$  where $\widetilde{w}=-\mathbf{i}\sin\theta+\mathbf{j}\cos\theta$. This seems to be the easiest way to develop the functions we wish. The area volume element is easy to find and so the result follows after a long series of computations for each function (reminiscent of the theory of ultra-spherical polynomials).
\end{proof}
The canonical push-forward of $\theta\wedge\alpha_2$ and $\alpha_0\wedge\alpha_2$ both vanish, but that of $\vol_\cals$ is ${4\pi}s^2\,\vol_M$. The proof is also an exercise and the result as expected.

Given any Riemannian vector bundle $E$ over $M$ and a section $\varphi\in\Gamma(M;T^*M\otimes E)$, we then have a real function on $\cals$ defined by $\tilde{\varphi}(u)=\varphi_{\pi(u)}(u),\ \forall u\in\cals$. It is easy to deduce that $({{\tilde{\varphi}}^2})^\vee=\frac{4\pi}{3}|\varphi|_{_{T^*M\otimes E}}^2$ (Hilbert-Schmidt norm).

For any section $g_1\in\Gamma(M;\otimes^2T^*M)$ on $M$, we may consider $({{\tilde{g}}_1^2})^\vee=\frac{4\pi}{3}|g_1|_{_{T^*M\otimes T^*M}}^2$ or otherwise, via the diagonal map, we find directly $(g_1^2)^\vee=\frac{4\pi}{3}\mathrm{tr}_gg_1$.

\subsection{Towards an intrinsic conservation law}

Let $(\cals,\theta)$ denote any contact manifold of dimension $2n+1$, equipped with a preferred contact form, such as the space we have been studying. Suppose it is given a differential ideal $\calj\subset\Omega^*_{\cals}$, where by \textit{differential} it is meant that $\dx\calj\subset\calj$. Then we may consider as in \cite{BGG} the exact sequence of complexes
\begin{equation}
 0\lrr\calj\lrr\Omega^*_{\cals}\lrr\Omega^*_{\cals}/\calj\lrr0
\end{equation}
and also the associated long exact sequence with field coefficients
\begin{equation}\label{seqexactalongadecohomologiasdeideais}
 \cdots H^{n-1}(\calj)  \lrr   H_{\mathrm{de\,R}}^{n-1}(\cals)  \lrr 
  H^{n-1}(\Omega^*_{\cals}/\calj)\lrr H^n(\calj)\lrr H_{\mathrm{de\,R}}^n(\cals) \cdots.
\end{equation}
In the event of the contact ideal $\cali=\{\theta,\dx\theta\}$ being contained in $\calj$, with integral submanifolds $f:N\rr\cals$ of dimension $n$, the real vector space $\calc=H^{n-1}(\Omega^*_{\cals}/\calj)$ is called the space of conservation laws (we assume the notation of the brackets referring just to the algebraic span in the exterior algebra of $\cals$). In other words, $\calc$ is the space of classes of $n-1$-forms $\varphi$ on $\cals$ such that $\dx f^*\varphi\in \calj$ for all integral submanifolds (\cite{BGG}). The contact ideal plays a central role. The contact condition $\theta\wedge(\dx\theta)^n\neq0$ implies that every $n+1$-form lies in $\cali$, giving e.g. the Poincar\'e-Cartan form of a given Lagrangian. For the same reason, an analogous condition holds with any other ideal $\calj$ containing the contact form $\theta$ and a non-degenerate 2-form over $\ker\theta$.

Finally we resume with the natural differential system $\theta,\alpha_0,\alpha_1,\alpha_2$ on the tangent sphere bundle $\cals$ of radius $s$ associated to any given oriented Riemannian 3-manifold $M$. A natural question is which intrinsic properties may there arise from the Euler-Lagrange system $\cale_\Lambda=\{\theta,\dx\theta,\Lambda\}$ when we take for the Lagrangian $\Lambda$ any of the invariant 2-forms. One may also study larger systems, including the ideal d-$\mathrm{span}$ of $\Gamma(\cals;W)\subset\Omega_\cals^2$ where $W=W_l,\ l=1,2,3$, from Proposition \ref{prop_decompwedgetwo}, or simply $W=\{e^1,e^2\}$ or $\{e^3,e^4\}$ recurring to any adapted frame. We notice that with a principal ideal we are more likely to find finite dimensions in \eqref{seqexactalongadecohomologiasdeideais}. Because we are interested in the $\SO(2)$-invariant 2-forms, we shall consider first the ideal $\calj=\{\Lambda\}$ generated by an invariant Lagrangian. The term \textit{invariant Lagrangian} is reserved here for any 2-form
\begin{equation}
 \Lambda=t_0\alpha_0+t_1\alpha_1+t_2\alpha_2+t_3\dx\theta
\end{equation}
such that $t_0,t_1,t_2,t_3\in\R$ are constants. We say that an invariant Lagrangian is \textit{degenerate} if $\Lambda\wedge\Lambda=0$. There is no preferred Lagrangian and many subclasses are quite important.
\begin{prop}
 $\Lambda$ is non-degenerate if and only if $t_0t_2-t_1^2-t_3^2\neq0$. On the subspace $\ker\theta$ we have (anti-)selfdual invariant Lagrangians, i.e. $*_4\Lambda=\pm\Lambda$, if and only if $t_2=\pm t_0,\ t_1=\mp t_1,\ t_3=\mp t_3$.
\end{prop}
From the structural equations (\ref{derivadasdastres2formas_alpha2}--\ref{dalpha0composto}), it follows that:
\begin{equation}
 \dx\Lambda=\theta\wedge\Lambda'_0+\Lambda'_1
\end{equation}
where
\begin{equation}
 \Lambda'_0=-rt_1\,\alpha_0+\frac{2t_0-s^2t_2r}{2s^2}\,\alpha_1
 +\frac{2t_1}{s^2}\,\alpha_2+t_2\,\gamma\qquad \mbox{and}\qquad\Lambda'_1=st_2\alpha_0\wedge\rho.
\end{equation}
Notice for every form $\tau$ there is a unique decomposition $\tau=\theta\wedge\tau_0+\tau_1$ where $\tau_1$ is free from factors of $\theta$.
Now we observe that a differential principal ideal may be defined from a 2-form $\Lambda$ such that $\dx\Lambda=\psi\wedge\Lambda$.
\begin{teo}
 Let $M$ be any oriented Riemannian 3-manifold $M$ and suppose $\Lambda$ is a non-degenerate invariant Lagrangian. Then $\dx\Lambda=\psi\wedge\Lambda$ if and only if one of the following conditions holds:\\
 i) $\Lambda\sim\dx\theta$;\\
 ii) $M$ has constant sectional curvature $c=\frac{t_0}{s^2t_2}$ and $\Lambda=\Lambda_1:=t_0\alpha_0+t_2\alpha_2+t_3\dx\theta$ is also closed, for any $t_0,t_2,t_3\in\R$ such that $t_0t_2\neq t_3^2$.
\end{teo}
The proof is elementary. Let us indicate by $\sim$ a real direct proportionals relation. Then with $\Lambda\sim\alpha_2$, which is degenerate, we have also a closed solution when $M$ is flat. The only solution with $\psi\sim\theta$ and $\psi$ non-vanishing is obtained through a degenerate Lagrangian. Precisely, it is defined on a negative constant sectional curvature $c=-\frac{t_0^2}{s^2}$ metric on $M$, for any non-vanishing $t_0$ and a degenerate Lagrangian proportional to 
\begin{equation}\label{selfPCLagrangian}
\Lambda_2:=t_0\alpha_0\pm\alpha_1+\frac{1}{t_0}\alpha_2 .
\end{equation}
This satisfies
\begin{equation}
 \dx\Lambda_2=\mp\frac{2t_0}{s^2}\,\theta\wedge\Lambda_2.
\end{equation}
\begin{lemma}\label{lemmadeintegridade}
 Let $e_0,\ldots,e_4$ be an adapted frame and let $\beta=\sum b_je^j$ denote any 1-form on $\cals$. Then $\beta\wedge\Lambda_2=0$ if and only if $b_0=t_0b_3\mp b_1=t_0b_4\mp b_2=0$.
\end{lemma}
Now we may study the cohomology $H^*(\Lambda)$, this is, the cohomology of the ideals spanned by the distinguished Lagrangians above. Of course $H^l(\Lambda)=H_{\mathrm{de\,R}}^{l}(\cals)$ for $l=0,1$.
\begin{prop}
We have:\\
i) $H^2(\dx\theta)=H^2(\Lambda_1)=\R$;\\
ii) $H^2(\alpha_2)=\{f\in\cinf{\cals}:\ \dx f\wedge\alpha_2=0\}$, in case $M$ is flat;\\
iii) $H^2(\Lambda_2)=\{f\in\cinf{\cals}:\ (\dx f\mp\frac{2t_0}{s^2}f\,\theta)\wedge\Lambda_2=0\}$, in the hyperbolic metric case above.
\end{prop}
\begin{proof}
 Clearly $H^2=Z^2=\{f\in\cinf{\cals}:\ \dx(f\Lambda)=0\}$ for any degree 2-form. In the two non-degenerate cases, we find $\dx f\wedge\Lambda$, for some function $f$ on $\cals$, vanishing if and only if $f$ is a constant. The remaining conditions are similar. For $\alpha_2$ the equation says $f$ does not vary horizontally.
\end{proof}
We notice that any non-trivial solution for case iii above should be quite interesting in the geometry of the hyperbolic base $M$. Of course Lemma \ref{lemmadeintegridade} is helpful but brings little insight to what kind of functions these are.


A next step in the theory is the study of the \textit{invariant} Euler-Lagrange systems, this meaning a differential ideal generated by an invariant Lagrangian $\Lambda$ and the contact 1-form:
\begin{equation}
 \cale_\Lambda=\{\theta,\dx\theta,\Lambda\}
\end{equation}

We shall end with an application, in extrinsic geometry, regarding the theory of calculus of variations and Legendre surfaces, cf. \cite{Alb2011arxiv,BGG}. We must see the interesting case of the \textit{degenerate system} given by $\Lambda_2$ above, which has as Poincar\'e-Cartan form essentially the form itself: $\dx\Lambda_2\sim\theta\wedge\Lambda_2$. Suppose $M$ has constant sectional curvature $c<0$. Recall the Gauss-Codazzi equation for a Riemannian hypersurface $f:N\lrr M$ reads $K_N=c+\lambda_1\lambda_2$, in the present dimension, where $\lambda_1,\lambda_2$ are the principal curvatures of $N$ and $K_N=R^N_{1212}$ is the sectional curvature. Also let $H_N=\frac{1}{2}(\lambda_1+\lambda_2)$ denote the mean curvature. Then we consider the following Weingarten type functional, for $t_0=\sqrt{-c}$:
 \begin{equation}
  \calf_{\Lambda_2}(N)=\int_N(K_N\mp 2t_0H_N+2t_0^2)\,\vol_N.
 \end{equation}
\begin{teo}
 Let $M$ be an oriented hyperbolic 3-manifold with constant sectional curvature $c$. Then a compact isometric immersed surface $f:N\rr M$ is stationary for the functional $\calf_{\Lambda_2}$ with fixed boundary if and only if
\begin{equation}
 K_N\mp 2t_0H_N+2t_0^2=0.
\end{equation}
\end{teo}
\begin{proof}
 Let $\hat{f}:N\lrr\cals_{M,1}$ denote the immersion induced by a unit normal vector field on $N$. Recalling \cite[Proposition 3.2]{Alb2011arxiv}, we see the pull-backs of the fundamental 2-forms $\alpha_0,\alpha_1,\alpha_2$ are a multiple of $\vol_N$ for the respective factors $1,-(\lambda_1+\lambda_2),\lambda_1\lambda_2$. Then $K_N\,\vol_N=\hat{f}^*(c\,\alpha_0+\alpha_2)$ and another straightforward computation shows  $\frac{1}{t_0}\int_N\hat{f}^*\Lambda_2$, cf. \eqref{selfPCLagrangian}, corresponds indeed to the functional defined by $\calf_{\Lambda_2}(N)$. Fundamental basics from \cite{BGG} yield that the stationary Legendre submanifolds are those which satisfy $\hat{f}^*\Psi=0$, when $\Pi=\theta\wedge\Psi$ is the Poincar\'e-Cartan form of the Euler-Lagrange system. In our case, $\Pi=\dx\Lambda_2\sim\theta\wedge\Lambda_2$.
\end{proof}

\section{New proof of the differential system in low dimensions}
\label{sec:Proofs}

The aim of this Section is to give a new proof of the fundamental differential system in dimensions 2 and 3.

\subsection{General computations}
\label{subsec:GeneralComputations}

We resume with the differential geometry considerations on the manifold $T_M$ endowed with the Sasaki metric and linear metric connection $\na^*$, reducible to a $\SO(n+1)$ connection, for any given Riemannian manifold $M$ of dimension $n+1$. As introduced in Section \ref{sec:Thedifferentialsystem}.

We continue to assume $\na$ is the Levi-Civita connection, so it is easy to give a torsion-free connection $D^*$ over $T_M$, cf. \eqref{nablasterisco}:
\begin{equation}
D^*_yz=\na^*_yz-\frac{1}{2}\calri(y,z),\quad \forall y,z\in TT_M .
\end{equation}
$D^*$ is most useful for many computations, though it is no longer a metric connection.

\vspace{3mm}

\noindent
\textsc{Remark.}
To find the Levi-Civita connection we must add to $D^*$ the tensor $A$ given by (cf. \cite{Alb2008,Alb2010,Alb2011,Alb2012,Alb2014a})
\begin{equation}
 \langle A_yz,w\rangle=\frac{1}{2}(\langle\calri(y,w),z\rangle+\langle\calri(z,w),y\rangle) .
\end{equation}
Recall $\calri$ is $V$-valued and notice $A$ is $H$-valued since $\calri(y,w)=\calri(y^h,w^h)$.
\vspace{3mm}

We shall work on the tangent bundle instead of its distinguished hypersurface $\cals$. It is  wiser to take restrictions only in the end. For the moment, we do not worry with $\cals$ and hence $s=\|\xi\|$ is a free parameter.

We may now prove formula \eqref{derivadade_rhogeral}. 
\begin{prop}\label{prop:derivadade_rhogeral}
 On $\cals$ we have $\dx\rho= 
 \frac{1}{s}\sum_{i=0}^ne^i\wedge\xi\lrcorner\na^*_i\pi^\star\ric$.
\end{prop}
\begin{proof}
 This computation is somewhat standard so we skip many details. First, after differentiation, we may disregard any factors of $\xi^\flat$, such as $\dx s=\frac{1}{s}\xi^\flat$, since these vanish on $\cals$. We then use the torsion-free $D^*$. It verifies, for any tensor form $L$ on $T_M$,
 \[  D^*_x(\xi\lrcorner L)=\xi\lrcorner(D^*_xL)+x^v\lrcorner L . \]
 We also have the expected symmetric tensor in $y,z$
 \[ D^*_x\pi^\star\ric(y,z)=\na^*_x\ric\,(y,z)+\frac{1}{2}\pi^*\ric(\calri_{x,y},z)+
 \frac{1}{2}\pi^*\ric(y,\calri_{x,z}).  \]
 Using all the symmetries involved to develop
 \[ \dx\rho=\frac{1}{s}\sum_{j=0}^{2n}e^j\wedge D^*_j(\xi\lrcorner\pi^\star\ric) \, \mod\xi^\flat, \]
 the result follows.
\end{proof}

Continuing with the adapted frame introduced in Section \ref{sec:Thedifferentialsystem}, we now recall that all the $n$-forms $\alpha_i$ recur to $\alpha_n$ and
\begin{equation}
\alpha_n\,=\,\frac{1}{s}\xi\lrcorner({\pi}^\star\vol_{_M})=
e^{(n+1)(n+2)\cdots(2n)} .
\end{equation}
Theorem \ref{derivadasdasnforms} is proved in \cite{Alb2011arxiv} with the tools of connection theory as introduced above. There, we differentiate the forms $\alpha_i$ applying an appropriate chain rule on the general definition \eqref{alpha_itravez}. We now come forward with a new study, we think also enlightening, of the $\alpha_i$, and we accomplish the task of finding their derivatives in dimensions 2 and 3. For higher dimensions, the new tools are still inquiring for one's talent, within the combinatorics required for the definitions, even knowing on the first place the expected result. We develop those ideas for $\dim M=n+1$ firstly and specialize with the low dimensions in the next subsections.

We need a lemma involving the tautological vector field $\xi$. For a moment, let $\xi$ denote just the position vector on Euclidean space. The next lemma proves the existence of a useful moving frame somewhat related to polar coordinates. Since we have not found it elsewhere, it is called here with the same name.
\begin{lemma}[Polar frame]
 For any $u_0\in\R^{n+1}\backslash\{0\}$ there exists a conical neighbourhood $U$ and a tangent frame $X_1,\ldots,X_n$ of $\xi^\perp$ defined on $U$, which on the line $\R u_0$ it is orthonormal and such that $(\partial_{X_j}X_i)_u=-\delta_{ij}\frac{u}{\|u\|^2}$, $\forall 1\leq i,j\leq n,\ \forall u\in\R u_0$. Moreover, everywhere on $U$, we have  $\partial_\xi X_i=0$ and $\partial_{X_i}\xi=X_i,\ \forall 1\leq i\leq n$.
\end{lemma}
\begin{proof}
Clearly $\partial_{X}\xi=X$ for every vector $X$. We take a normal chart on the radius 1 $n$-sphere passing through $u_1=u_0/\|u_0\|$. Such a coordinate system is critical for the Levi-Civita connection $\na^\sigma$ with maximal rank at the centre $u_1$, i.e., the Christoffel symbols vanish at $u_1$. Of course, we may suppose the induced frame $X_1,\ldots,X_n$ to be orthonormal at $u_1$. Then we lift the vectors to the product manifold $S^n\times\R$. In other words, by Euclidean parallel translation along the ray. Immediately we have $\partial_\xi X_i=0$ and $X_i\perp\xi$ on $U$. Now the crucial point is that at $u_0$ we still have vanishing Christoffel symbols. Indeed, homotheties preserve the sphere geodesics and at the centre the scale does not change those values. Finally
 \[ 0=\na^\sigma_{X_i}X_j=\partial_{X_i}X_j-\frac{1}{\|\xi\|^2}\langle\partial_{X_i}X_j,\xi\rangle\xi =\partial_{X_i}X_j+\frac{\delta_{ij}}{\|\xi\|^2}\xi   \]
and the result follows.
\end{proof}
A simple example is enough to reassure the factors are correct.
\vspace{2mm}
\\\textsc{Example.} In $\R^2$ we have $\xi_{(x,y)}=(x,y)$ and then take $X_{(x,y)}=\frac{1}{s}(-y,x)$ with $s=\sqrt{x^2+y^2}$. Clearly $\dx s(X)=0$ and $(\partial_XX)_{(x,y)}=-\frac{1}{s^2}(x,y)$. Also $\partial_\xi X=0$. Notice that while this result is global, that in the lemma is local --- because normal coordinates depend on a chosen basis for $n>1$. The same is to say, in $n$ distinct great circles.
\vspace{2mm}

We return to $T_M$ and its linear connections $\na^*$ and $D^*$. The tautological vector field verifies $\na^*_\xi\xi=\xi$ by \eqref{nablaxi}. Also recall $\|\xi\|=s$.
\begin{prop}\label{existenceofadaptedpolarframe}
 For all non-vanishing $u_0\in T_M$ there is a neighbourhood $U$ of $u_0$ and a vertical frame $e_{n+1},\ldots,e_{2n}$ of $V\cap\xi^\perp$ defined on $U$, such that on the line $\R u_0$ it is orthonormal and such that $\na^*_{e_{j+n}}e_{i+n}=-\delta_{ij}\frac{\xi}{s^2}$, $\forall 1\leq i,j\leq n$. Everywhere on $U$ we have that $\na^*_\xi e_{i+n}=0$ and $\na^*_{e_{i+n}}\xi=e_{i+n},\ \forall 1\leq i\leq n$.
\end{prop}
\begin{proof}
 Around any point $x_0=\pi(u_0)\in M$ there is a neighbourhood $W$ domain of a trivialization of $T_M$ and a smooth vector field $\hat{u}$ defined on $W$ and passing through $u_0$. Using the lemma above and the smooth dependence on initial conditions (the vector field $\hat{u}$) of the normal coordinates used in the proof above, we find the desired frame on the trivialization domain.
\end{proof}
In the next step we take the horizontal mirror of the vertical polar frame and thus find on the neighbourhood $U\subset T_M$ an \textit{adapted polar frame}: $e_0=\frac{1}{s} B^{\mathrm{t}}\xi,\,e_1,\ldots,e_n$, $\frac{1}{s}\xi,e_{n+1},\ldots,e_{2n}$. On the horizontal directions we have, for some general matrix 1-form $\omega$ defined on $U$, the usual formula $\na^*_{e_i}e_j=\sum_{k=0}^n\omega_{ij}^ke_k$.
\begin{prop}\label{adaptedpolarframeproposition}
At point $u_0$ from Proposition \ref{existenceofadaptedpolarframe}, the resulting covariant derivatives of the adapted frame are as follows (let $i,j=1,\ldots,n$):
\begin{equation}\label{adaptedpolarframe}
\begin{array}{cccc}
\na^*_{{0}}B^{\mathrm{t}}\xi=0 , &
\na^*_{{0}}e_j=\sum_{k=1}^n\omega_{0j}^ke_{k} , &
\na^*_{{0}}e_{i+n}=\sum_{k=1}^n\omega_{0i}^ke_{k+n} , &
\na^*_{{0}}\xi=0 , \\
\na^*_{{i}}B^{\mathrm{t}}\xi=0 , &
\na^*_{{i}}e_j=\sum_{k=1}^n\omega_{ij}^ke_k , &
\na^*_{{i}}e_{j+n}=\sum_{k=1}^n\omega_{ij}^ke_{k+n} , &
\na^*_{{i}}\xi=0 , \\
\na^*_{{i+n}}B^{\mathrm{t}}\xi=e_i , &
\na^*_{{i+n}}e_j=-\delta_{ij}\frac{e_0}{s} , &
\na^*_{{i+n}}e_{j+n}=-\delta_{ij}\frac{\xi}{s^2} , &
\na^*_{{i+n}}\xi=e_{i+n} , \\
\na^*_{{\xi}}B^{\mathrm{t}}\xi=B^{\mathrm{t}}\xi , &
\na^*_{{\xi}}e_j=0 , &
\na^*_{{\xi}}e_{i+n}=0 , &
\na^*_{{\xi}}\xi=\xi .
\end{array}
\end{equation}
Moreover, $\forall w\in TT_M$,
\begin{equation}
 \na^*_w\frac{1}{s}\xi=\frac{1}{s}w^v-\frac{1}{s^3}\xi^\flat(w)\xi,\qquad
 \na^*_we_0=\frac{1}{s}B^\mathrm{t}w-\frac{1}{s^2}\xi^\flat(w)e_0.
\end{equation}
\end{prop}
\begin{proof}
 This is a consequence of $\na^*B=0$ and Proposition \ref{existenceofadaptedpolarframe}. Also notice  $\langle\na^*_ke_{j},e_0\rangle=0$, $\forall k=0,\ldots,n$, which explains why the four sums start at 1. For the last two formulae we have $\dx s=\frac{1}{s}\xi^\flat$ and hence $\dx\frac{1}{s}=-\frac{1}{s^3}\xi^\flat$. The result follows very easily. 
\end{proof}
We shall need the following formula putting the curvature in terms of horizontals.
\begin{prop}
$\forall x,y\in TT_M$, 
\begin{equation}
 D^*_xy^\flat =(\na^*_xy)^\flat+ 
          \frac{s}{2}\sum_{k=0}^n\langle R^\na(x^h,e_k)e_0,B^\mathrm{t}y^v\rangle e^k.
\end{equation}
\end{prop}
\begin{proof}
Indeed,
\begin{align*}
 (D^*_xy^\flat)z & =x(\langle y,z\rangle)-\langle y,D^*_xz\rangle  \\
  & =\langle\na^*_xy,z\rangle+\frac{1}{2}\langle y,\calri(x,z)\rangle  \\
  & =\langle\na^*_xy,z\rangle+ 
  \frac{s}{2}\langle \pi^\star R^\na(x^h,z^h)\frac{\xi}{s},y^v\rangle \\
  & =\bigl((\na^*_xy)^\flat+\frac{s}{2}\sum_{k=0}^n\langle R^\na(x^h,e_k)e_0,B^\mathrm{t}y^v\rangle e^k\bigr)z .
\end{align*}
\end{proof}
The following is now easy to check.
\begin{prop}
 In the conditions of Proposition \ref{existenceofadaptedpolarframe}, we have:
\begin{equation}
\begin{array}{cc}
  D^*_{{0}}e^0=0 , & D^*_{{0}}e^j=\sum_{k=1}^n\omega_{0j}^ke^k , \\
  D^*_{{i}}e^0=0 , &  D^*_{{i}}e^j=\sum_{k=1}^n\omega_{ij}^ke^k , \\
  D^*_{{i+n}}e^0=\frac{1}{s}e^i , & D^*_{{i+n}}e^j=-\frac{\delta_{ij}}{s}e^0 , \\
  D^*_{{\xi}}e^0=0 , & D^*_{{\xi}}e^j=0 ,
\end{array}
\end{equation}
and
\begin{equation}\label{adaptedpolarframeDderivatives}
\begin{array}{cc}
D^*_{{0}}e^{i+n}=\sum_{k=1}^n\omega_{0i}^ke^{k+n}+\frac{s}{2}\sum_{k=1}^nR_{i00k}e^k , &
D^*_{{0}}\frac{1}{s}\xi^\flat=0 , \\
D^*_{{i}}e^{j+n}=\sum_{k=1}^n\omega_{ij}^ke^{k+n}+\frac{s}{2}\sum_{k=0}^nR_{j0ik}e^k , &
D^*_{{i}}\frac{1}{s}\xi^\flat=0 , \\
D^*_{{i+n}}e^{j+n}=-\frac{\delta_{ij}}{s^2}\xi^\flat , &
D^*_{{i+n}}\frac{1}{s}\xi^\flat=\frac{1}{s} e^{i+n} , \\
D^*_{{\xi}}e^{i+n}=0 , &
D^*_{{\xi}}\frac{1}{s}\xi^\flat=0.
\end{array}
\end{equation}
\end{prop}
A simple consequence is yet another way to compute the derivative of the contact form $\theta$, cf. \eqref{mu} and \cite{Alb2011arxiv}. Indeed, before restriction to the tangent sphere bundle, the 1-form $s\,e^0$ is the metric parallel equivalent to the natural Liouville form of the cotangent bundle. Using the torsion free connection, the new method yields
\begin{equation}
  \dx(s\,e^0)=\frac{1}{s}\xi^\flat\wedge e^0+s\sum_{k=0}^{2n}e^k\wedge D^*_ke^0+\frac{s}{s^2}\xi^\flat\wedge D^*_\xi e^0=\frac{1}{s}\xi^\flat\wedge e^0+\sum_{k=n+1}^{2n}e^{k,k-n}.
\end{equation}
Clearly, when we pull-back by the inclusion map $\cals\hookrightarrow T_M$ we obtain $\theta$ and the known formula of $\dx\theta$. Interesting enough, notice $\dx(\frac{1}{s}\xi^\flat)=\dx\dx s=0$ and $\dx\xi^\flat=\frac{1}{2}\dx\dx s^2=0$.

Now, for $1\leq i\leq n$, we have
\begin{align}
 \dx e^i &=\sum_{k=0}^{2n}e^k\wedge D^*_ke^i+\frac{1}{s^2}\xi^\flat\wedge D^*_\xi e^i \nonumber \\
   &= \sum_{k,j=0}^n\omega_{ji}^ke^{jk}-\sum_{k=1}^{n}\frac{\delta_{ki}}{s}e^{(k+n)0} \nonumber \\
   &=\frac{1}{s}e^{0(i+n)}+\sum_{j,k=0}^n\omega_{ji}^ke^{jk}
\end{align}
and
\begin{align}
 \dx e^{i+n} &=\frac{s}{2}\sum_{j=1}^nR_{i00j}e^{0j}+ \sum_{j,k=0}^n\omega_{ji}^ke^{j(k+n)}+\frac{s}{2}
 \sum_{j=1,k=0}^nR_{i0jk}e^{jk}-e^{i+n}\frac{\xi^\flat}{s^2} \nonumber \\
 &= \sum_{j=1}^n(sR_{i00j}e^{0j}+\omega_{0i}^je^{0(j+n)})+ \sum_{j,k=1}^n(\omega_{ji}^ke^{j(k+n)}+\frac{s}{2}
 R_{i0jk}e^{jk})-\frac{1}{s^2}e^{i+n}\xi^\flat .
\end{align}
Of course these formulae are valid at any point of $\R u_0\subset T_M$. On this generic line, centre of an adapted polar frame, we have $\omega_{ij}^k=-\omega_{ik}^j$ and $\omega_{ij}^0=0$.

\subsection{On Riemannian 2-manifolds}

In case $M$ has dimension 2, this is, $n=1$, we have a global coframing of  $\cals$ with $\theta=s\,e^0$ and two 1-forms $\alpha_0=e^1$ and $\alpha_1=e^2$ pulled-back by the inclusion map of the circle in the plane tangent bundle of $M$. Moreover, the circle bundle agrees with a principal $\SO(2)$ frame bundle. Still over $T_M$ we have
\begin{align*}
  &\qquad\qquad  \dx e^1=\frac{1}{s}e^{02}+\omega_{01}^1e^{01}
   +\omega_{11}^0e^{10}=\frac{1}{s}e^{02},\\
 & \dx e^2=sR_{1001}e^{01}+\omega_{01}^1e^{02}+\omega_{11}^1e^{12}
-\frac{1}{s^2}e^2\xi^\flat =sR_{1001}e^{01}-\frac{1}{s^2}e^2\xi^\flat .
\end{align*}
The following formulae, where $c=R_{1010}$ denotes Gauss curvature, consist of the  First and Second Cartan Structural Equations in dimension 2 using the well-known terminology. After restriction, on $\cals$ we have found:
\begin{equation}\label{dalphais_m=2_em geral}
 \dx\theta=\alpha_1\wedge\alpha_0 ,\qquad \dx\alpha_1=c\,\alpha_0\wedge\theta,\qquad   \dx\alpha_0=\frac{1}{s^2}\,\theta\wedge\alpha_1  .
\end{equation}
Together with the general proof given in \cite{Alb2011arxiv} and that in \cite[pp. 168--169]{SingerThorpe}, there are now three independent proofs of Theorem \ref{derivadasdasnforms} for Riemannian 2-manifolds.

\subsection{On Riemannian 3-manifolds}

Back in the case $n=2$ we recall $\alpha_0=e^{12}$, $\alpha_1=e^{14}-e^{23}$, $\alpha_2=e^{34}$. As above, these forms are previous and freely defined on the tangent manifold, where we have:
\begin{align}
 \dx e^{12} &=(\dx e^1)e^2-e^1\dx e^2 \nonumber \\
    &=\frac{1}{s}e^{032}+\sum_{j,k=0}^2\omega_{j1}^ke^{jk2}
    -\frac{1}{s}e^{104}-\sum_{j,k=0}^2\omega_{j2}^ke^{1jk} \nonumber \\
    &=\frac{1}{s}\,e^0(e^{14}-e^{23})
\end{align}
and 
\begin{align*}
 \dx(e^{14}-e^{23}) &=(\dx e^1)e^4-e^1\dx e^4 -(\dx e^2)e^3+e^2\dx e^3 \\
 &= \frac{1}{s}e^{034}+\sum_{j,k=0}^2\omega_{j1}^ke^{jk4}
 -s\sum_{j=1}^2R_{200j}e^{10j}-\sum_{k=1}^2\omega_{02}^ke^{10(k+2)}\\
 &\qquad -\sum_{j,k=1}^2(\omega_{j2}^ke^{1j(k+2)}+\frac{s}{2}R_{20jk}e^{1jk}) 
 +\frac{1}{s^2}e^{14}\xi^\flat -\frac{1}{s}e^{043}  \\
 &\qquad -\sum_{j,k=0}^2\omega_{j2}^ke^{jk3}+s\sum_{j=1}^2R_{100j}e^{20j}
 +\sum_{k=1}^2\omega_{01}^ke^{20(k+2)} \\
 &\qquad +\sum_{j,k=1}^2(\omega_{j1}^ke^{2j(k+2)}+\frac{s}{2}R_{10jk}e^{2jk}) 
 -\frac{1}{s^2}e^{23}\xi^\flat\qquad\mbox{(cont.),}
 \end{align*}
notice this time the cancellation of omegas happens in pairs,
\begin{align}
 &= \frac{2}{s}e^{034}+\omega_{01}^2e^{024}+\omega_{11}^2e^{124}
 -sR_{2002}e^{102} \nonumber \\ &\qquad-\omega_{02}^1e^{103}-\omega_{22}^1e^{123} 
 +\frac{1}{s^2}(e^{14}-e^{23})\xi^\flat -\omega_{02}^1e^{013}  \nonumber\\
 &\qquad-\omega_{22}^1e^{213}+sR_{1001}e^{201}+\omega_{01}^2e^{204}
 +\omega_{11}^2e^{214} \nonumber \\
 &=  \frac{2}{s}e^{034}-s(R_{2020}+R_{1010})e^{012}
 +\frac{1}{s^2}(e^{14}-e^{23})\xi^\flat
\end{align}
and
\begin{align}
 \dx e^{34} &=(\dx e^3)e^4-e^3\dx e^4 \nonumber \\
  &= s\sum_{j=1}^2R_{100j}e^{0j4}+\sum_{j,k=0}^2\omega_{j1}^ke^{j(k+2)4}+\frac{s}{2}\sum_{j,k=1}^2R_{10jk}e^{jk4}  -\frac{1}{s^2}e^3\xi^\flat e^4 \nonumber \\
  &\qquad-\sum_{j=1}^2(sR_{200j}e^{30j}+
  \omega_{02}^je^{30(j+2)})+\frac{1}{s^2}e^{34}\xi^\flat \nonumber\\
  &\qquad -\sum_{j,k=1}^2(\omega_{j2}^ke^{3j(k+2)}
  +\frac{s}{2}R_{20jk}e^{3jk}) \nonumber \\
  &= sR_{1001}e^{014}+sR_{1002}e^{024}+\frac{s}{2}(R_{1012}e^{124}+R_{1021}e^{214})+\frac{2}{s^2}e^{34}\xi^\flat \nonumber \\
  &\qquad -\frac{s}{2}(R_{2012}e^{312}+R_{2021}e^{321})-sR_{2001}e^{301}-sR_{2002}e^{302} \nonumber \\ &= s\,e^0(R_{1001}e^{14}+R_{1002}e^{24}
  +R_{2001}e^{31}+R_{2002}e^{32}) \nonumber \\
  &\qquad +sR_{1012}e^{124}-sR_{2012}e^{123}+\frac{2}{s^2} e^{34}\xi^\flat
  \label{derivadada2formalpha0}
\end{align}
The pull-back to $\cals$ of the three 2-forms above and their derivatives on $T_M$ clearly have the desired form, found, respectively, in \eqref{dalpha0porextenso}, \eqref{derivadasdastres2formas_alpha1} and \eqref{derivadasdastres2formas_alpha2}. Thus a new independent proof of Theorem \ref{derivadasdasnforms} in dimension 3 is achieved.
\vspace{2mm}
\\\textsc{Remark.}
We recall there is also a proof in \cite{Alb2010} for \textit{flat} Euclidean space in dimension 4 using a global moving frame on $\R^4\times S^3$. For the interested reader we recall here the general 3-forms for case $n=3$. They are $\alpha_0=e^{123}$, $\alpha_1=e^{126}+e^{234}+e^{315}$, $\alpha_2=e^{156}+e^{264}+e^{345}$ and $\alpha_3=e^{456}$.

\medskip


\ 

\ 

\ 

\textsc{R. Albuquerque}

{\small\texttt{rpa@uevora.pt}}

\

Centro de Investiga\c c\~ao em Mate\-m\'a\-ti\-ca e Aplica\c c\~oes

Rua Rom\~ao Ramalho, 59, 7000-671, \'Evora, Portugal

\ \\
The research leading to these results has received funding from the People Programme (Marie Curie Actions) of the European Union's Seventh Framework Programme (FP7/ 2007-2013) under REA grant agreement n\textordmasculine~PIEF-GA-2012-332209.


\end{document}